\documentclass[12pt,twoside,reqno]{amsart}
\usepackage[colorlinks=true,citecolor=blue]{hyperref}
\usepackage{mathptmx, amsmath, amssymb, amsfonts, amsthm, mathptmx, enumerate, color,mathrsfs}
\usepackage{graphicx}

\usepackage{multirow}
\usepackage{epstopdf}
\usepackage{multicol}
\usepackage{algorithm}
\usepackage{algorithmic}
\usepackage{epstopdf}

\newtheorem{theorem}{Theorem}[section]
\newtheorem{lemma}[theorem]{Lemma}
\newtheorem{proposition}[theorem]{Proposition}

\theoremstyle{definition}

\newtheorem{example}[theorem]{Example}

\newtheorem{remark}[theorem]{Remark}

\numberwithin{equation}{section}

\begin{document}
\setcounter{page}{1}

\vspace*{2.0cm}
\title[Diffusive solutions and QM mappings]
{A note on diffusive solutions of the Lyapunov and Riccati inequalities for quasi-monotone (QM) mappings on cones}
\author[O. Mason]{O. Mason$^{1,*}$}
\maketitle
\vspace*{-0.6cm}

\begin{center}
{\footnotesize

$^1$Department of Mathematics \& Statistics, Maynooth University, Co. Kildare, Ireland \\

\vspace{5mm}

Dedicated to Professor Ezra Zeheb on the occasion of his 85th birthday.

}\end{center}

\vskip 4mm {\footnotesize \noindent {\bf Abstract.}
We consider three key properties of Metzler and nonnegative matrices and extensions of these to classes of self-dual proper convex cones.  Specifically, we study mappings that are quasi-monotone (QM) with respect to a cone $K$ and discuss results extending D-stability, diagonal Lyapunov stability, and diagonal Riccati stability to this setting.  Mappings that act diffusively with respect to the cone are used as generalisations of diagonal matrices.  Relationships with recent results for symmetric cones obtained using Jordan algebraic methods are also discussed.


  }

\renewcommand{\thefootnote}{}
\footnotetext{ $^*$Corresponding author.
\par
E-mail address: oliver.mason@mu.ie.
\par
Received xx, x, xxxx; Accepted xx, x, xxxx.

\rightline {\tiny   \copyright  2022 Communications in Optimization Theory}}

\section{Introduction}

We consider generalisations of three classical properties of positive linear time-invariant (LTI) systems.  It is well known that the system $\dot{x} = Ax$ is positive if and only if the matrix $A$ is Metzler, meaning $a_{ij} \geq 0$ for $i \neq j$: such matrices are also sometimes referred to as \emph{cross-positive}, \emph{quasi-monotone}, or as the negative of Z matrices.  In the classical theory, positivity is understood with respect to the nonnegative orthant $\mathbb{R}^n_+$.  Stability of the system is equivalent to $A$ being Hurwitz: having all its eigenvalues in the open left half plane.  We shall refer to such matrices as stable or Hurwitz throughout.  The first two results we consider show that stable positive LTI systems are diagonally stable \cite{narsho} and D stable \cite{kush}.  The former means that there exists a positive definite diagonal solution $D$ to the Lyapunov inequality $A^TD+DA \prec 0$, while the latter means that for any positive definite diagonal $D\succ 0$, $DA$ is Hurwitz.    

The stability theory for positive LTI systems has been extended in several directions, such as switched positive linear systems \cite{durbundfuss, masonshorten}, infinite dimensional systems \cite{miron}, and time-delay systems \cite{Haddad, briat}.  The third result we consider concerns the last of these system classes.  A linear time-delay system $\dot{x} = Ax + B x(t -\tau)$ is positive if and only if $A$ is Metzler and $B$ is nonnegative.  It was shown in \cite{Haddad} that if $A+B$ is Hurwitz, then the delayed system is stable for any delay $\tau \geq 0$.  Later, inspired by a question posed by Verriest in \cite{SysProbs}, it was shown that under these same assumptions, there exists a \emph{diagonal} Lyapunov-Krasosvkii functional.  In matrix theoretic terms, this means that there exist diagonal matrices $D \succ 0$, $Q\succ 0$ satisfying 
\begin{equation}\label{eq:Ricc0}
\left(\begin{array}{c c}
A^TD+DA + Q & B^TD \\
DB & -Q 
\end{array}\right) \prec 0.
\end{equation}  
An alternative, simple proof of this result, as well as extensions to more general positive time-delay systems can be found in the paper \cite{briat}.

Systems that are positive with respect to more general cones than $\mathbb{R}^n_+$ have also been studied, and it is this direction that we will consider here.  Interesting results on copositve Lyapunov functions for switched linear systems positive with respect to cones defined by linear inequalities can be found in \cite{durbundfuss}.  Z-transformations on proper and symmetric cones were studied in \cite{gowda1} and several classical results for Z matrices \cite{BermanNeuSter} were derived in this setting.  In particular, a result on diagonal stability for Z transformations on symmetric cones was derived, using the framework of Euclidean Jordan algebras, with the \emph{quadratic representations} of the algebra playing the role of diagonal matrices.  More recently, in \cite{gowda2}, cone-preserving (nonnegative) Z transformations were studied and characterised for proper and symmetric cones.  Interestingly, it was shown that for irreducible cones, all cone-preserving Z transformations are given by nonnegative multiples of the identity, while they are characterised as 'nonnegative block diagonal' transformations for reducible cones.  In the recent paper \cite{LamRicc}, an interesting extension of the result of \cite{mason2012} was derived for time-delay systems leaving a symmetric cone invariant.  This result can also be viewed as an extension of that on diagonal Lyapunov stability in \cite{gowda1} to the case of diagonal Riccati stability.  As with the papers \cite{gowda1}, \cite{gowda2}, the machinery of Euclidean Jordan algebras was used, and the quadratic representation played the role of diagonal matrices.  

In this brief note, inspired by the results mentioned above, we consider a slight variation on the themes of the previous paragraph.  In particular, we will discuss diagonal Lyapunov and Riccati stability, and D stability for \emph{quasi-monotone} (QM) mappings on a Euclidean space equipped with a proper, self-dual cone: these mappings are a natural generalisaton of Metzler matrices and we give a formal definition of them in the following section. For some of our main results, we impose an extra technical assumption which is closely related to the homogeneity assumption in the definition of a symmetric cone.  In contrast to the powerful results in \cite{gowda1, LamRicc} for symmetric cones, we do not make any explicit use of Jordan algebraic techniques.  Also, we work with a different generalisation of diagonal matrices, which was introduced in the recent paper \cite{deLeen}.  The results of \cite{deLeen} were motivated by stability question for diffusively coupled systems arising in Ecology and linear systems with an invariant proper cone were considered.  At appropriate points in the paper, we highlight how the results here relate to those mentioned in the paragraph above.  
   
\section{Preliminaries}
Throughout, we consider a finite dimensional Euclidean space $V$, equipped with an inner product $\langle \cdot, \cdot \rangle$.  The space of linear mappings from $V$ to itself is denoted by $L(V)$.  The adjoint mapping of $A \in L(V)$ is denoted by $A^*$.  We use the notation $P \succ 0$ ($P\succeq 0$) to denote that the self-adjoint $P$ is positive definite (semi-definite).  The notations $P \prec 0$, $P\preceq 0$ have the obvious meaning.  We use $\mathcal{S}^n$ and $\mathcal{S}^n_+$ to denote the space and cone (respectively) of $n\times n$ real symmetric matrices and real positive semi-definite matrices.

A subset $K \subseteq V$ is a convex cone if $\alpha x  + \beta y \in K$ for all $x, y \in K$, $\alpha, \beta > 0$.  A proper cone is a convex cone $K \subseteq  V$ that is closed, pointed (meaning $K \cap (-K) = \{0\}$), and generating ($V = K-K$).   Denote by $\textrm{Aut}(K)$ the linear automorphisms of $K$, consisting of all invertible linear transformations $T$ of $V$ with $T(K)=K$.  If $\textrm{Aut}(K)$ acts transitively on $K$, $K$ is homogeneous. 

The dual cone $K^*$ of $K$ is given by $K^*=\{y \in V: \langle x, y \rangle \geq 0 \mbox{ for all }x \in K\}$.  $K$ is self-dual if $K=K^*$.  A proper cone is symmetric if it is self-dual and homogeneous.  If $K \subseteq V$ is a symmetric cone, then a product can be defined on $V$ such that $V$ with this product becomes a Euclidean Jordan algebra and $K$ is given by the cone of squares in $V$: formally $K = \{x^2 : x \in V\}$.  For background on symmetric cones and Euclidean Jordan algebras, the interested reader should consult the monograph \cite{FarKor}.  

Given a proper, self-dual, cone $K$ in $V$, we say that a linear mapping $T\in L(V)$ is quasi-monotone (QM) with respect to $K$ if $\langle x,y \rangle = 0$, $x,y \in K$ implies $\langle x, T y \rangle \geq 0$.  These mappings have been referred to as cross-positive in other sources: for example in \cite{LamRicc}.  It is not difficult to show that for $V = \mathbb{R}^n$, $K = \mathbb{R}^n_+$, a mapping given by the matrix $A \in \mathbb{R}^{n \times n}$ is QM with respect to $K$ if and only if $A$ is Metzler.  A mapping $T \in L(V)$ is $K$-nonnegative if $T(K) \subseteq K$.  As noted in \cite{LamRicc}, it is straightforward to show that if $A$ is QM on $K$ and $B$ is $K$-nonnegative, then $A+B$ is also QM on $K$.

Motivated by a basic question arising in theoretical ecology, QM mappings were recently studied in \cite{deLeen} in the context of diffusively coupled maps.  The key question considered was to determine stability conditions for a coupled system where the coupling between component systems was defined by a family of \emph{diffusive} mappings.  Diffusive mappings generalise diagonal matrices and it is this generalisation that we will consider here.  The mapping $D \in L(V)$ acts diffusively with respect to a proper cone $K$ if $D(K) \subseteq K$, and $x,y \in K$, $\langle x, y \rangle = 0$ implies $\langle x, D y \rangle = 0$.  

In keeping with standard terminology, we say that a mapping $T$ in $L(V)$ is stable if all the eigenvalues of $T$ have negative real part.  The following result gives a characterisation of stability for QM mappings and generalises well known results about Metzler matrices.  A proof can be found in \cite{deLeen, gowda1}, for example.

\begin{proposition}\label{prop:stab1}
Let $K\subseteq V$ be a proper cone and $A \in L(V)$ be QM with respect to $K$.  Then $A$ is stable if and only if there exists some $v \in \textrm{int}(K)$ with $-Av \in \textrm{int}(K)$. 
\end{proposition} 

We will use the notation $x > y$ for $x-y \in K\setminus \{0\}$, $x \gg y$ for $x-y \in \textrm{int}(K)$.  So the previous result states that for a QM $A$, stability is equivalent to the existence of $v \gg 0$ with $Av \ll 0$.  Note that $A$ is stable if and only if its adjoint $A^*$ is stable so another equivalent condition is the existence of $w \gg 0$ with $A^* w \ll 0$.  

\section{Main results}
Throughout this section, we assume that $K$ is a proper, self-dual cone in the Euclidean space $V$.  We will present a number of results concerning QM mappings and mappings that act diffusively with respect to $K$.      

\begin{remark}
For $V=\mathbb{R}^n$ and $K=\mathbb{R}^n_+$, mappings acting diffusively with respect to $K$ are simply given by nonnegative diagonal matrices.  
More generally, consider a \emph{proper, self-dual polyhedral cone}, $K$ with finite generating set $\{x_1, \ldots, x_n\}$.  Given nonnegative real numbers $d_1, \ldots d_n$, define $D(\alpha_1 x_1 + \cdots + \alpha_n x_n) = d_1 \alpha_1 x_1 + \cdots + d_n \alpha_n x_n$ on $K$. This naturally extends to a linear mapping on $V$ as $V=K-K$.  Then it is immediate that $D(K) \subseteq K$.  Moreover, suppose $x,y \in K$ and $\langle x,y \rangle = 0$.  Then $y=\sum_{i=1}^n \alpha_i y_i$ for some nonnegative real numbers $\alpha_i$, $1\leq i \leq n$, and we have
\[\sum_{i=1}^n \alpha_i \langle x, y_i\rangle = 0.\]
As $x, y_i$ are in $K$ and $K$ is self-dual, $\langle x, y_i \rangle \geq 0$ for all $i$ and $\alpha_i \geq 0$ for all $i$.  Thus, it follows that for $1\leq i \leq n$, either $\alpha_i = 0$ or $\langle x, y_i \rangle = 0$.  This implies that 
\[\langle x, Dy \rangle = \sum_{i=1}^n \alpha_i d_i \langle x , y_i\rangle = 0.\]
So, $D$ acts diffusively on $K$.
\end{remark}

\begin{remark}
As noted in \cite{deLeen}, the set of mappings acting diffusively on a cone is always non-empty as it contains all nonnegative multiples of the identity.  It has been shown in \cite{gowda2} that if the cone $K$ is irreducible, then any Z transformation that leaves $K$ invariant is a nonnegative multiple of the identity.  Recall that a Z transformation is the negative of a QM mapping.  It is clear that any mapping acting diffusively with respect to $K$ is a Z transformation and leaves $K$ invariant.  For this reason, in order to have an interesting class of mappings with which to work, it is necessary to work with reducible cones such as the self-dual polyhedral cones discussed above.  
\end{remark} 

\subsection{Diagonal and D-stability}
We begin with some simple lemmas.  
\begin{lemma}
\label{lem:QMadj}  Let $A \in L(V)$ be QM with respect to $K$.  Then $A^*$ is also QM with respect to $K$.
\end{lemma} 
\textbf{Proof:} Suppose $x,y$ in $K$ satisfy $\langle x,y \rangle = 0$.  Then $\langle y, x \rangle = 0$ and, as $K$ is self-dual, and $A$ is QM, it follows that 
\[\langle x, A^*y \rangle = \langle Ax, y \rangle = \langle y, Ax \rangle \geq 0. \]
Thus $A^*$ is also QM.

\begin{lemma}
\label{lem:QMD}  Let $A \in L(V)$ be QM with respect to $K$ and $D \in L(V)$ act diffusively on $K$.  Then $DA$ and $A^*D$ are both QM with respect to $K$.
\end{lemma} 
\textbf{Proof:} Let $x, y$ in $K$ satisfy $\langle x, y \rangle = 0$.  Then as $D$ acts diffusively on $K$ $\langle x, Dy \rangle = 0$ and $Dy \in K$.  As $A$ is QM, Lemma \ref{lem:QMadj} implies that $A^*$ is QM.  Hence $\langle x, A^*Dy \rangle \geq  0$.  This shows that $A^*D$ is QM on $K$.  Another application of Lemma \ref{lem:QMadj} shows that $(A^*D)^* = D^*A$ is QM.  Note that if $D$ acts diffusively on $K$, then so does $D^*$.  Thus, replacing $D$ with $D^*$, $DA = (D^*)^* A$ is also QM.

\begin{lemma}\label{lem:Diffint}
Let $E \in L(V)$ act diffusively on $K$.  Suppose $E$ is invertible.  Then $E(\textrm{int}(K)) \subseteq \textrm{int}(K)$. 
\end{lemma}
\textbf{Proof:} Let $x \in \textrm{int}(K)$.  Then $Ex \in K$, and for any $y \in K \setminus \{0\}$, 
\[\langle Ex,y \rangle = \langle x, E^*y \rangle.\]
As $E$ and hence $E^*$ is invertible, it follows that $E^* y \int K \setminus \{0\}$, so $\langle x, E^*y \rangle > 0$ as $x \in \textrm{int}(K)$.  This shows that $\langle Ex,y \rangle > 0$ for any $y \in K \setminus \{0\}$ and hence $Ex \in \textrm{int}(K)$ as claimed.  

The next result extends the D stability property for stable Metzler matrices and positive LTI systems on $\mathbb{R}^n$ to the setting of proper self-dual cones.  
\begin{proposition}\label{prop:D}
Let $A \in L(V)$ be QM with respect to $K$ and stable.  Then $EA$ is QM with respect to $K$ and stable for any invertible $E \in L(V)$ that acts diffusively on $K$.
\end{proposition}
\textbf{Proof:} 
It follows immediately from Lemma \ref{lem:QMD} that $EA$ is QM.  From Proposition \ref{prop:stab1}, there exists a $v \in \textrm{int}(K)$ with $-Av \in \textrm{int}(K)$.  As $E$ acts diffusively on $K$ and is invertbile, Lemma \ref{lem:Diffint} shows that $-EA v \in \textrm{int}(K)$.  Proposition \ref{prop:stab1} immediately implies that $EA$ is stable as claimed.  

In the next proposition, we also make the following key assumption about the cone $K$.  

\textbf{Assumption D}  For any $v, w$ in $\textrm{int}(K)$, there exists a self-adjoint mapping $D$, diffusive on $K$, with $Dv =w$.  Note that $D$ is not formally assumed to satisfy $D(K) = K$ here so, on the surface at least, this is a distinct assumption to the assumption of homogeneity for symmetric cones.  

For a simple example of a cone $K$ satisfying Assumption D, consider $V=\mathbb{R}^n$ with the usual Euclidean inner product, and take $K=O(\mathbb{R}^n_+)$ to be the image of the nonnegative orthant under an orthogonal matrix $O$.  Then $v, w$ are in the interior of $K$ if and only if $v=Ox$, $w=Oy$ for $x, y$ in the interior of $\mathbb{R}^n_+$.  Let $E$ be the unique diagonal matrix mapping $x$ to $y$.  Set $D=OEO^*$.  Then, it is easy to verify that $D$ acts diffusively on $K$, is self-adjoint, and that $Dv=(OEO^*)(Ox) = OEx=Oy=w$.  

\begin{proposition}\label{prop:DDiag}
Assume that $K$ satisfies Assumption D.  Let $A \in L(V)$ be QM with respect to $K$ and stable.  Then there exists a symmetric $D \succ 0$ in $L(V)$ that acts diffusively on $K$ with $A^*D + DA \prec 0$.
\end{proposition}
\textbf{Proof:} It follows from Lemma \ref{lem:QMadj} that $A$ and $A^*$ are QM and stable.  Hence, there exist $v, w$ in $\textrm{int}(K)$ with $-Av, -A^*w$ in $\textrm{int}(K)$.  Using Assumption D, choose some self-adjoint $D$ that acts diffusively on $K$ such that $Dv=w$.  Now note that
\[-(A^*D+DA)v= -A^*(Dv) - DA v = -A^*w- D(Av).\]
By construction, $-A^*w \in \textrm{int}(K)$ as is $-Av$.  As $D$ acts diffusively on $K$, $-D(Av) = D(-Av)$ is in $K$.  So, $-A^*w \in \textrm{int}(K)$ and $-D(Av)$ is in $K$ and hence $-A^*w- D(Av) \in \textrm{int}(K)$.  Lemma \ref{lem:QMD} implies that $A^*D+DA$ is QM, so it is stable by Proposition \ref{prop:stab1}.  As it is self-adjoint, this immediately implies that $A^*D+DA \prec 0$.  

\subsection{Riccati stability}
The question of Riccati stability was posed by Verriest in \cite{SysProbs} along with some partial results.  For the special case of positive time delay systems, corresponding to a Metzler matrix $A$ and nonnegative matrix $B$, the following result was proven in \cite{mason2012}.  An alternative, simple proof can be found in \cite{briat}, while a general characterisation of diagonal Riccati stability was given in \cite{masonalek2016}.    

Recently, in the spirit of the results of \cite{gowda1} for Lyapunov stability, Lam and Shen proved an interesting extension of this for positive time-delay systems on Euclidean Jordan algebras.  
\begin{proposition}
\label{prop:Ricc}  Let $A, B$ in $L(V)$ be given.  Assume that $K \subseteq V$ is a proper self-dual cone satisfying Assumption D.  Suppose that $A$ is QM on $K$, $B$ is $K$-nonnegative and $A+B$ is stable.  Then there exist $D$, $Q$ in $L(V)$ that act diffusively on $K$ such that 
\begin{equation}\label{eq:Ricc}
\left(\begin{array}{c c}
A^*D+DA + Q & B^*D \\
DB & -Q 
\end{array}\right) \prec 0
\end{equation}  
\end{proposition} 

\textbf{Proof:} As $A$ is QM and $B$ is $K$-nonnegative, $A+B$ is QM on $K$.  Proposition \ref{prop:DDiag} implies that there exists some $D \in L(V)$, acting diffusively on $K$, such that 
\[(A+B)^*D+D(A+B) \prec 0.\]
As in the proof of Proposition \ref{prop:DDiag}, we can show that $(A+B)^*D +D (A+B)$ is QM on $K$.  Therefore $(A+B)^*D+D(A+B)$ is QM and stable, and it follows from Proposition \ref{prop:stab1} that there exists some $v \in \textrm{int}(K)$ with $-((A+B)^*D +D (A+B))v = w$ in $\textrm{int}(K)$.  As $D$ and $B$ are both $K$-nonnegative by assumption, $DBv \in K$ and $DBv + \frac{1}{2} w \in \textrm{int}(K)$.  Assumption D implies that there exists some $Q$ that acts diffusively on $K$ with $Qv=DBv+\frac{1}{2} w$.  

Now consider the block linear operator \[M=\left(\begin{array}{c c}
A^*D+DA + Q & B^*D \\
DB & -Q 
\end{array}\right)\]
on the direct sum $V^2=V\oplus V$.  The set $K \times K = \{(x,y) \in V^2:x, y \in K\}$ is clearly a proper, self-dual cone in $V^2$, with respect to the inner product $\langle (x,y), (s,t) \rangle = \langle x, s\rangle + \langle y, t \rangle$.  Suppose that $(x,y)$, $(s,t)$ are in $K \times K$ and that \begin{equation}\label{eq:innprod2} \langle (x,y), (s,t) \rangle = \langle x, s\rangle + \langle y, t \rangle = 0.\end{equation}
As all of $x,y,s,t$ are in $K$, which is self-dual, $\langle x, s\rangle \geq 0$, $\langle y,t\rangle \geq 0$.  Thus, it follows from \eqref{eq:innprod2} that $\langle x,s \rangle =\langle y , t\rangle = 0$.  By direct calculation, we can see that 
\begin{eqnarray*}
\langle (x,y), M (s,t)\rangle &=& \langle x, (A^*D+DA+Q) s +B^*D t \rangle + \langle y, PB s -Qt  \rangle \\
&=& \langle x, (A^*D+DA+Q) s \rangle + \langle x, B^*D t \rangle + \langle y, DB s \rangle - \langle y, Qt \rangle.
\end{eqnarray*}
From Lemma \ref{lem:QMD} and the fact that $Q$ acts diffusively on $D$, it is easy to see that $A^*D+DA+Q$ is QM with respect to $K$.  Thus, as $\langle x,s \rangle=0$, it follows that $ \langle x, (A^*D+DA+Q) s \rangle \geq 0$.  $\langle x, B^*D t \rangle \geq 0$ as $x, t \in K$ and $B^*D$ is clearly $K$-nonnegative.  Similarly,  $\langle y, DB s \rangle \geq 0$.  Finally, as $\langle y, t \rangle = 0$ and $Q$ acts diffusively on $K$, $ \langle y, Qt \rangle = 0$.  Putting all of this together, we have that $\langle (x,y), M (s,t)\rangle \geq 0$.  Thus, $M$ is QM on $K \times K$.  

To finish the proof, we just need to verify that $-M (v, v) \in \textrm{int}(K)$.  This follows from observing that $M(v,v) = (-w+ Qv-DB v, DB v - Qv) = \frac{1}{2}(-w, -w)$.  Thus $M$ is QM and stable by Proposition \ref{prop:stab1}.  As it is clearly self-adjoint, this implies that $M \prec 0$, which completes the proof. 

\subsection{Comments on relation to Jordan algebraic results}
The papers \cite{gowda1} and \cite{LamRicc} contain results on the existence of diagonal solutions to the Lyapunov and Riccati inequalities for mappings that are QM with respect to symmetric cones.  As mentioned in the Introduction, any symmetric cone in a finite dimensional Euclidean space can be realised as the cone of squares in a Euclidean Jordan algebra \cite{FarKor}.  These are finite dimensional algebras, where the product is non-associative.  Classical examples include the spaces $\mathcal{S}^n$ of symmetric $n\times n$ matrices equipped with the product $A\circ B = \frac{1}{2}(AB+BA)$.  In this short subsection, we make some remarks on how the results in the current paper relate to these.  

In a Euclidean Jordan algebra, $V$, the quadratic representation $P_a$ for $a \in V$ is defined by $P_a(x) = 2(a(ax))-a^2(x)$.  These mappings can be seen as generalisations of diagonal matrices, and play a key role in the results of \cite{gowda1,LamRicc}.  Theorem 11 of \cite{gowda1} shows that if $A$ is stable and QM with respect to the symmetric cone $K$ of $V$, then there exists $a$ in the interior of $K$ such that $A^*P_a + P_a A \prec 0$.  Similarly, the main result, Theorem 1, of \cite{LamRicc} shows that for $A$ QM and $B$ nonnegative with respect to $K$, if $A+B$ is stable, there exists a pair $(D,Q)$ satisfying \eqref{eq:Ricc} with $D$ and $Q$ given by $P_v$, $P_w$ for appropriately chosen elements $v$, $w$ of the interior of $K$.  For general Euclidean Jordan algebras, the mappings $P_a$ do not act diffusively with respect to the cone of squares $K$.  This is unsurprising as such mappings are QM on $K$ and $K$ nonnegative.  So, the results of \cite{gowda2} show that when the cone is irreducible the only possibilities are nonnegative multiples of the identity.  More explicitly, consider the space $\mathcal{S}^n$ with cone $\mathcal{S}^n_+$ of positive semi-definite matrices.  The quadratic representation is defined by $X \mapsto AXA$, and the inner product is $\langle X, Y \rangle = \textrm{trace}(XY)$.  Considering
\[X=\left(\begin{array}{c c}
1 & 0 \\
0 & 0 
\end{array}\right), 
Y=\left(\begin{array}{c c}
0 & 0 \\
0 & 1 
\end{array}\right),
A=\left(\begin{array}{c c}
1 & 1 \\
1 & 1 
\end{array}\right),
\]
it is easy to verify that $\langle X, Y \rangle = 0$ but $\langle X, AYA \rangle > 0$.  

If we consider the \emph{reducible} Euclidean Jordan algebra $R^n$ equipped with the elementwise product, the cone of squares $K$ is the usual nonnegative orthant, which is reducible.  Here the quadratic representation corresponds to multiplication by nonnegative diagonal matrices and these do act diffusively on $K$.  

%Also, if $(x,y)$, $(u,v)$ are in the interior of $K\times K$, then $x$, $y$, $u$, $v$ must lie in the interior of $K$ and hence we can find mappings $D_1$, $D_2$ acting diffusively on $K$ such that $D_1 x = u$, $D_2 y = v$.  Defining the mapping $D$ on $V^2$ by $D(x,y) = (D_1x, D_2 y)$, we can verify that $D$ acts diffusively with respect to $K \times K$.  
%\subsection{Monotonicity, Riccati differential equations, and symmetric cones}
%RESULT FROM DRESDEN

\subsection{A negative point about diagonal Riccati stability for switched systems}
Finally, we revisit the classical setting of $\mathbb{R}^n_+$ and a question concerned with common diagonal solutions to pairs of Riccati inequalities associated with positive time-delay systems.  One of the more noteworthy aspects of positive systems is that many stability properties are preserved under the addition of time-delay.  The results in \cite{mason2012, briat} on diagonal solutions to the Riccati inequality established that the property of diagonal stability for a single system, is preserved under the addition of time-delay.  If we consider a pair of such systems, characterised by two pairs of matrices $(A_1,B_1)$, $(A_2, B_2)$ with $A_i$ Metzler, $B_i$ nonnegative for $i=1,2$, the following question seems natural.  If there exists a common diagonal solution, $E$ to the Lyapunov inequalities $(A_i+B_i)^TE+E(A_i+B_i) \prec 0$, $i=1,2$, does there exist a common diagonal solution $(D,Q)$ $D\succ 0, Q \succ 0$ to the Riccati inequalities 
\[A_i^*D+DA_i + DB_iQ^{-1}B_iD + Q \prec 0?\]
Unfortunately, this appealing property does not hold in general, as shown by the following numerical counterexample, found using the MATLAB LMI toolbox. 
\begin{example}
\label{ex:CDRFctr}
\[A_1= \left(\begin{array}{c c}
-1.4093  & 0.1501 \\
  0.0986  & -1.3504
\end{array}\right), 
B_1= \left(\begin{array}{c c}
0.7743 & 0.1205 \\
0.6820 &  0.7193
\end{array}\right)\]
\[A_2= \left(\begin{array}{c c}
-0.5474 & 0.0626\\
    0.1537 & -1.1340
\end{array}\right), 
B_2= \left(\begin{array}{c c}
0.1619 & 0.6812 \\
0.1202 & 0.1448
\end{array}\right).\]
It can be verified that 
\[E=\left(\begin{array}{c c}
0.7266 &  0\\
 0 & 0.5828
\end{array}\right)\]
satisfies $A_i^TE+EA_i \prec 0$ for $i=1,2$.  However, it can be checked via the LMI toolbox that there is no common solution $(D,Q)$ with diagonal $D, Q \succ 0$ to 
\begin{equation}
\left(\begin{array}{c c}
A_i^TD+DA_i + Q & B_i^TD \\
DB_i & -Q 
\end{array}\right) \prec 0
\end{equation} 
for $i=1,2$.  
\end{example} 
\vskip 6mm
\newpage
\noindent{\bf Acknowledgements}

\noindent
The author would like to thank the anonymous reviewer for their helpful and constructive feedback, and to thank the editors for their patience and the kind invitation to contribute to this collection.

\end{document}